%% file: holomorphic_quadratic_arxiv.tex
\documentclass[a4paper,final]{amsart}

\usepackage{graphicx}

\usepackage{amsmath}
\usepackage{amsfonts}
\usepackage{amssymb}
\usepackage{amsthm}

\usepackage{esvect}
\usepackage{tikz}

\theoremstyle{plain}
\newtheorem{theorem}{Theorem}[section]
\newtheorem{lemma}[theorem]{Lemma}
\newtheorem{corollary}[theorem]{Corollary}
\theoremstyle{definition}
\newtheorem{remark}[theorem]{Remark}
\newtheorem{definition}[theorem]{Definition}
\newtheorem{example}[theorem]{Example}

\DeclareMathOperator{\T}{T}
\DeclareMathOperator{\Arg}{Arg}
\DeclareMathOperator{\cratio}{cr}

\DeclareMathOperator{\Real}{Re}
\DeclareMathOperator{\Imaginary}{Im}
\DeclareMathOperator{\grad}{grad}

\renewcommand{\Re}{\Real}

\renewcommand{\Im}{\Imaginary}

\begin{document}
	
	\title[]{Holomorphic vector fields and quadratic differentials on planar triangular meshes}
	\author{Wai Yeung Lam}
	\author{Ulrich Pinkall}
	
	\address{Wai Yeung Lam\\
		Technische Universit\"at Berlin\\Institut f\"ur Mathematik\\
		Stra{\ss}e des 17.\ Juni 136\\
		10623 Berlin\\ Germany}
	
	\address{Ulrich Pinkall\\
		Technische Universit\"at Berlin\\Institut f\"ur Mathematik\\
		Stra{\ss}e des 17.\ Juni 136\\
		10623 Berlin\\ Germany}
	
	\email{lam@math.tu-berlin.de, pinkall@math.tu-berlin.de}

	\begin{abstract}{
				Given a triangulated region in the complex plane, a discrete vector field $Y$ assigns a vector $Y_i\in \mathbb{C}$ to every vertex. We call such a vector field holomorphic if it defines an infinitesimal deformation of the triangulation that preserves length cross ratios. We show that each holomorphic vector field can be constructed based on a discrete harmonic function in the sense of the cotan Laplacian. Moreover, to each holomorphic vector field we associate in a M\"obius invariant fashion a certain holomorphic quadratic differential. Here a quadratic differential is defined as an object that assigns a purely imaginary number to each interior edge. Then we derive a Weierstrass representation formula, which shows how a holomorphic quadratic differential can be used to construct a discrete minimal surface with prescribed Gau{\ss} map and prescribed Hopf differential.
		}
	\end{abstract}
	
	\thanks{This research was supported by the DFG Collaborative Research Centre SFB/TRR 109 \emph{Discretization in Geometry and Dynamics}.}
	
	\date{\today}
	
	\maketitle
	
	\section{Introduction}
	Consider an open subset $U$ in the complex plane $\mathbb{C}\cong \mathbb{R}^2$ with coordinates $z=x+iy$ together with a holomorphic vector field
	\begin{equation*}
		Y=f\frac{\partial}{\partial x}.
	\end{equation*}
	Here $Y$ is a real vector field. It assigns to each $p\in \mathbb{R}^2$ the vector $f(p)\in\mathbb{C}\cong \mathbb{R}^2$. We do not consider objects like $\frac{\partial}{\partial z}$ which are sections of the complexified tangent bundle $(\T\mathbb{R}^2)^{\mathbb{C}}$.
	
	Note $f:U \to \mathbb{C}$ is a holomorphic function, i.e.
	\begin{equation*}
		0=f_{\bar{z}}=\frac{1}{2}\left(\frac{\partial f}{\partial x}+i\frac{\partial f}{\partial y}\right).
	\end{equation*}
	Let $t \mapsto g_t$ denote the local flow of $Y$ (defined for small $t$ on open subsets of $U$ with compact closure in $U$). Then the euclidean metric pulled back under $g_t$ is conformally equivalently to the original metric:
	\begin{equation*}
		g_t^*\langle \,,\rangle=e^{2u}\langle \,,\rangle
	\end{equation*}
	for some real-valued function $u$. The infinitesimal change in scale $\dot{u}$ is given by
	\begin{equation*}
		\dot{u}=\frac{1}{2}\mbox{div}\,Y = \mbox{Re}\left(f_z\right).
	\end{equation*}
	Note that $\dot{u}$ is a harmonic function:
	\begin{equation*}
		\dot{u}_{z\bar{z}}=0.
	\end{equation*}
	On the other hand, differentiating $\dot{u}$ twice with respect to $z$ yields one half the third derivative of $f$:
	\begin{equation*}
		\dot{u}_{zz}=\frac{1}{2}f_{zzz}.
	\end{equation*}
	It is well-known that the vector field $Y$ corresponds to an infinitesimal M\"obius transformation of the extended complex plane $\overline{\mathbb{C}}$ if and only if $f$ is a quadratic polynomial. In this sense $f_{zzz}$ measures the infinitesimal ``change in M\"obius structure'' under $Y$ (M\"obius structures are sometimes also called ``complex projective structures'' \cite{Gunning1966}). Moreover, the holomorphic quadratic differential
	\begin{equation*}
		q:=f_{zzz}\,dz^2
	\end{equation*}
	is invariant under M\"obius transformations $\Phi$. This is equivalent to saying that $q$ is unchanged under a change of variable $\Phi(z)=w=\xi+i\eta$ whenever $\Phi$ is a M\"obius transformation. This is easy to see if $\Phi(z)=az+b$ is an affine transformation. In this case
	\begin{align*}
		dw&=a\,dz\\
		\frac{d}{dw}&=\frac{1}{a}\frac{d}{dz}
	\end{align*}
	and therefore
	\begin{equation*}
		Y=\tilde{f}\frac{\partial}{\partial \xi}
	\end{equation*}
	with
	\begin{equation*}
		\tilde{f}=a\,f.
	\end{equation*}
	Thus we indeed have
	\begin{equation*}
		\tilde{f}_{www}\,dw^2=f_{zzz}\,dz^2.
	\end{equation*}
	A similar argument applies to $\Phi(z)=\frac{1}{z}$ and therefore to all M\"obius transformations.
	
	For realizations from an open subset $U$ of the Riemann sphere $\mathbb{C}\textrm{P}^1$ the vanishing of the Schwarzian derivative characterizes M\"obius transformations. The quadratic differential $q$  plays a similar role for vector fields. We call $q$ the {\em M\"obius derivative} of $Y$.
	
	An important geometric context where holomorphic quadratic differentials arise comes from the theory of minimal surfaces: Given a simply connected Riemann surface $M$ together with a holomorphic immersion $g:M \to S^2 \subset \mathbb{R}^3$ and a holomorphic quadratic differential $q$ on $M$, there is a minimal surface $F:M\to \mathbb{R}^3$ (unique up to translations) whose Gau{\ss} map is $g$ and whose second fundamental form is $\mbox{Re}\,q$.
	
	In this paper we will provide a discrete version for all details of the above story. Instead of smooth surfaces we will work with triangulated surfaces of arbitrary combinatorics. The notion of conformality will be that of conformal equivalence as explained in \cite{Bobenko2010}. Holomorphic vector fields will be defined as infinitesimal conformal deformations.
	
	There is also a completely parallel discrete story where conformal equivalence of planar triangulations is replaced by preserving intersection angles of circumcircles. To some extent we also tell this parallel story that belongs to the world of circle patterns.
	
	The results on planar triangular meshes in this paper are closely related to isothermic triangulated surfaces in Euclidean space \cite{Lam2015}.
	
	\section{Discrete conformality}
	
	In this section, we review two notions of discrete conformality for planar triangular meshes. We first start with some notations of triangular meshes.
	
	\begin{definition}
		A triangular mesh $M$ is a simplicial complex whose underlying topological space is a connected 2-manifold (with boundary). The set of vertices (0-cells), edges (1-cells) and triangles (2-cells) are denoted as $V$, $E$ and $F$.
	\end{definition}
	
	We denote $E_{int}$ the set of interior edges and $V_{int}$ the set of interior vertices. Without further notice we will assume that all triangular meshes under consideration are oriented.
	
	\begin{definition}
		A {\em realization} $z:V \to \mathbb{C}$ of a triangular mesh $M$ in the extended complex plane assigns to each vertex $i\in V$ a point $z_i \in \overline{\mathbb{C}}$ in such a way that for each triangle $\{ijk\}\in F$ the points corresponding to its three vertices are not collinear.
	\end{definition}
	
	Given two complex numbers $z_1,z_2 \in \mathbb{C}$ we write
	\[
	\langle z_1, z_2 \rangle := \Re( \bar{z}_1 z_2).
	\]
	
	We are looking for suitable definitions of conformal structure of a realization $z$. In particular, we want $z$ to be conformally equivalent to $g\circ z$ whenever $g: \overline{\mathbb{C}} \to \overline{\mathbb{C}}$ is a M\"{o}bius transformations. This requirement will certainly be met if we base our definitions on complex cross ratios: Given a triangular mesh $z:V\to \mathbb{C}$, we  associate a complex number to each interior edge $\{ij\} \in E_{int}$, namely the \emph{cross ratio} of the corresponding four vertices (See Figure \ref{fig:orientation})
	\[
	\cratio_{z,ij} = \frac{(z_{j} - z_k)(z_i - z_l)}{(z_k - z_i)(z_l - z_{j})}.
	\]
	Notice that $\cratio_{z,ij}=\cratio_{z,ji}$ and hence $\cratio_z:E_{int} \to \mathbb{C}$ is well defined.
	\begin{figure}[h]
		\centering
		\def\svgwidth{0.4\textwidth}
		\resizebox{0.4\textwidth}{!}{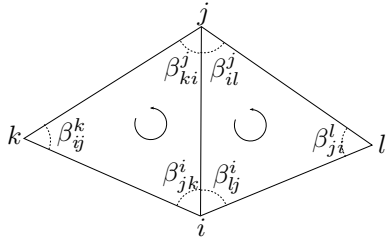}
		\caption{Two neighboring and oriented triangles}
		\label{fig:orientation}
	\end{figure}
	It is easy to see that two realizations differ only by a M\"{o}bius transformation if and only if their corresponding cross ratios are the same. In order to arrive at a more flexible notion of conformality we need to relax the condition that demands the equality of all cross ratios. Two natural ways to do this is to only require equality of either the norm or alternatively the argument of the cross ratios. This leads to two different notions of discrete conformality: \textit{conformal equivalence theory} \cite{Luo2004,Springborn2008} and \textit{circle pattern theory} \cite{Schramm1997}. 
	
	Note that for the sake of simplicity of exposition we are ignoring here realizations in $\overline{\mathbb{C}}$ where one of the vertices is mapped to infinity.
	
	\subsection{Conformal equivalence}
	
	The edge lengths of a triangular mesh realized in the complex plane provide a discrete counterpart for the induced Euclidean metric in the smooth theory. A notion of conformal equivalence based on edge lengths was proposed by Luo \cite{Luo2004}. Later Bobenko et al. \cite{Bobenko2010} stated this notion in the following form:	
	\begin{definition}
		Two realizations of a triangular mesh $z,w:V \to \mathbb{C}$ are \emph{conformally equivalent} if the norm of the corresponding cross ratios are equal:
		\[|\cratio_{z}| \equiv |\cratio_{w}|,\]
		i.e. for each interior edge $\{ij\}$ 
		\begin{equation*}
			\frac{|(z_{j} - z_k)||(z_i - z_l)|}{|(z_k - z_i)||(z_l - z_{j})|} = \frac{|(w_{j} - w_k)||(w_i - w_l)|}{|(w_k - w_i)||(w_l - w_{j})|}.
		\end{equation*}
	\end{definition}
	
	This definition can be restated in an equivalent form that closely mirrors the notion of conformal equivalence of Riemannian metrics:
	\begin{theorem} \label{thm:conformalequivalence}
		Two realizations of a triangular mesh $z,w:V \to \mathbb{C}$ are conformally equivalent if and only if there exists $u:V\rightarrow \mathbb{R}$ such that
		\[
		|w_{j} -w_i|=e^{\frac{u_{i}+u_{{j}}}{2}}|z_{j} -z_i|.
		\]
	\end{theorem}
	\begin{proof}
		It is easy to see that the existence of $u$ implies conformal equivalence. Conversely, for two conformally equivalent realizations $z,w$, we define a function $\sigma: E \to \mathbb{R}$ by
		\[
		|w_{j} -w_i|=e^{\sigma_{ij}}|z_{j} -z_i|.
		\]
		Since $z,w$ are conformally equivalent $\sigma$ satisfies for each interior edge $\{ij\}$
		\[
		\sigma_{jk}-\sigma_{ki}+\sigma_{il}-\sigma_{lj}=0.
		\]
		For any vertex $i$ and any triangle $\{ijk\}$ containing it we then define
		\[
		e^{u_i} := e^{\sigma_{ki} + \sigma_{ij} - \sigma_{{j}k}}. 
		\]
		Note the vertex star of $i$ is a triangulated disk if $i$ is interior, or is a fan if $i$ is a boundary vertex. Hence the value $u_i$ defined in this way is independent of the chosen triangle.  		
	\end{proof}
	
	\subsection{Circle patterns}
	
	Given a triangular mesh realized in the complex plane we consider the circumscribed circles of its triangles. These circles inherit an orientation from their triangles. The intersection angles of these circles from neighboring triangles define a function $\phi : E_{int} \to [0, 2\pi)$ which is related to the argument of the corresponding cross ratio via
	\begin{equation} \label{eq:intersectionangles}
		e^{i \phi_{ij}} = \Arg(\cratio_{z,ij}).
	\end{equation}
	
	\begin{figure}[h]
		\centering
		\def\svgwidth{0.4\textwidth}
		\resizebox{0.4\textwidth}{!}{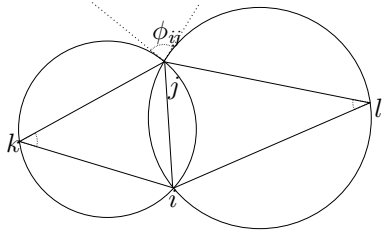}
		\caption{The intersection angle of two neighboring circumscribed circles}
	\end{figure}
	
	Based on these angles we obtain another notion of discrete conformality which reflects the angle-preserving property that we have in the smooth theory.
	\begin{definition}
		Two realizations of a triangular mesh $z,w:V \to \mathbb{C}$ have the same \emph{pattern structure} if the corresponding intersection angles of neighboring circumscribed circles are equal: 
		\[
		\Arg(\cratio_{z,ij})=\Arg(\cratio_{w,ij}),
		\]
		i.e. for each interior edge $\{ij\}$
		\[
		\Arg \frac{(z_{j} - z_k)(z_i - z_l)}{(z_k - z_i)(z_l - z_{j})} = \Arg \frac{(w_{j} - w_k)(w_i - w_l)}{(w_k - w_i)(w_l - w_{j})}.
		\]
	\end{definition}
	
	Just as conformal equivalence was related to scale factors $u$ at vertices, having the same pattern structure is related to the existence of certain angular velocities $\alpha$ located at vertices:
	
	\begin{theorem}
		Two realizations of a triangular mesh $z,w:V \to \mathbb{C}$ have the same pattern structure if and only if there exists $\alpha:V \rightarrow [0, 2\pi )$ such that
		\[
		\frac{w_{j} -w_i}{|w_{j} -w_i|}=e^{i\frac{\alpha_{i}+\alpha_{j}}{2}}\frac{z_{j} -z_i}{|z_{j} -z_i|}.
		\]
	\end{theorem}
	\begin{proof}
		The argument is very similar to the one for Theorem \ref{thm:conformalequivalence}. In particular, the existence of the function $\alpha$ easily implies equality of the pattern structures. Conversely, assuming identical pattern structures we take any $\omega:E \to \mathbb{R}$ that satisfies		
		\[
		\frac{w_{j} - w_i}{|w_{j} - w_i|} =  e^{i\omega_{ij}} \frac{z_{j} - z_i}{|z_{j} - z_i|} .
		\]
		For any vertex $i$ and any triangle $\{ijk\}$ containing it we define $\alpha_i \in [0, 2\pi )$ such that
		\[
		e^{i \alpha_i} = e^{i (\omega_{ki} + \omega_{ij} - \omega_{jk})}.
		\]
		Note the vertex star of $i$ is a triangulated disk if $i$ is interior, or is a fan if $i$ is a boundary vertex. Hence having the same pattern structure implies that the value $\alpha_i$ is independent of the chosen triangle.  
	\end{proof}
	
	\section{Infinitesimal deformations and linear conformal theory}
	
	We will linearize both of the above notions of discrete conformality by considering infinitesimal deformations. This will allow us to relate them to linear discrete complex analysis, based on a discrete analogue of the Cauchy Riemann equations \cite{Ferrand1944,Duffin1956,Mercat2001} (See the survey \cite{Smirnov2010}).	
	\begin{definition} \label{def:infcon}
		An \emph{infinitesimal conformal deformation} of a realization $z:V \to \mathbb{C}$ of a triangular mesh is a map $\dot{z}: V \to \mathbb{C}$ such that there exists $u:V \to \mathbb{R}$ satisfying 
		\[
		\Re{\frac{\dot{z}_{j} - \dot{z}_i}{z_{j}-z_i}} = \frac{\langle \dot{z}_{j}-\dot{z}_i, z_{j} - z_i \rangle}{|z_{j}-z_i|^2} = \frac{u_i+u_{j}}{2}.
		\]
		We call $u$ the \emph{scale change} at vertices.
	\end{definition}
	
	\begin{definition} \label{def:infpatt}
		An \emph{infinitesimal pattern deformation} of a realization $z:V \to \mathbb{C}$ of a triangular mesh is a map $\dot{z}: V \to \mathbb{C}$ such that there exists $\alpha:V \to \mathbb{R}$ satisfying
		\[
		\Im{\frac{\dot{z}_{j} - \dot{z}_i}{z_{j}-z_i}} = \frac{\langle \dot{z}_{j}-\dot{z}_i, i(z_{j} - z_i) \rangle}{|z_{j}-z_i|^2} = \frac{\alpha_i+\alpha_{j}}{2}.
		\]
		We call $\alpha$ the \emph{angular velocities} at vertices.
	\end{definition}
	
	\begin{example}
		The infinitesimal deformations $\dot{z}:= a z^2 + bz + c$, where $a,b,c \in \mathbb{C}$ are constants, are both conformal and pattern deformations since
		\[
		\frac{\dot{z}_{j} - \dot{z}_i}{z_{j}-z_i} =(az_i + b/2) + (az_{j} + b/2).
		\]
	\end{example}
	
	Infinitesimal conformal deformations and infinitesimal pattern deformations are closely related:
	
	\begin{theorem}
		Suppose $z:V \to \mathbb{C}$ is a realization of a triangular mesh. Then an infinitesimal deformation $\dot{z}:V \to \mathbb{C}$ is conformal if and only if $i \dot{z}$ is a pattern deformation. 
	\end{theorem}
	\begin{proof}
		Notice
		\begin{align*}
			\frac{\langle \dot{z}_{j}-\dot{z}_i, z_{j} - z_i \rangle}{|z_{j}-z_i|^2} =  \frac{\langle i\dot{z}_{j}- i\dot{z}_i, i(z_{j} - z_i) \rangle}{|z_{j}-z_i|^2}.
		\end{align*}
		and the claim follows from Definition \ref{def:infcon} and \ref{def:infpatt}.  
	\end{proof}

	\subsection{Infinitesimal deformations of a triangle}
	Let $z:V \to \mathbb{C}$ be a realization of a triangulated mesh and $\dot{z}$ an infinitesimal deformation. Up to an infinitesimal translation $\dot{z}$ is completely determined by the infinitesimal scalings and rotations that it induces on each edge. These infinitesimal scalings and rotations of edges satisfy certain compatibility conditions on each triangle. These conditions involve the cotangent coefficients well known from the theory of discrete Laplacians. As we will see in section \ref{sec:harmonic}, for conformal deformations (as well as for pattern deformations) the infinitesimal scalings and rotations of edges are indeed discrete harmonic functions.
	
	Consider three pairwise distinct points $z_1,z_2,z_3 \in \mathbb{C}$ that do not lie on a line. In the following $i,j,k$ denotes any cyclic permutation of the indexes $1,2,3$. The triangle angle at the vertex $i$ is denoted by $\beta_i$. We adopt the convention that all $\beta_1,\beta_2,\beta_3$ have positive sign if the triangle $z_1,z_2,z_3$ is positively oriented and a negative sign otherwise. Suppose we have an infinitesimal deformation of this triangle. Then there exists $\sigma_{ij},\omega_{ij} \in \mathbb{R}$ such that
	\begin{equation}
		\label{eq:zdot}
		\dot{z}_{j} - \dot{z}_i = (\sigma_{ij} + i\omega_{ij}) (z_{j} -z_i).
	\end{equation}
	The scalars $\sigma_{ij}$ and $\omega_{ij}$ describe the infinitesimal scalings and rotations of the edges. They satisfy the following compatibility conditions:
	
	\begin{lemma} \label{lem:infindeform}
		
		Given $\sigma_{ij},\omega_{ij} \in \mathbb{R}$ the following statements are equivalent:
		
		(a) There exist $\dot{z}_i$ such that (\ref{eq:zdot}) holds.
		
		(b) We have		\begin{equation} \label{eq:closedtri}
			0 = (\sigma_{12} + i\omega_{12})(z_2 - z_1) + (\sigma_{23} + i\omega_{23})(z_3 - z_2) +(\sigma_{31} + i\omega_{31})(z_1 -z_3).
		\end{equation}
		
		(c) There exists $\omega \in \mathbb{R}$ such that
		\begin{align*}
			i\omega &= i\omega_{23} + i \cot \beta_1 (\sigma_{31} - \sigma_{12}) \nonumber\\
			&= i\omega_{31} + i \cot \beta_2 (\sigma_{12} - \sigma_{23}) \\
			&= i\omega_{12} + i \cot \beta_3 (\sigma_{23} - \sigma_{31}) \nonumber.
		\end{align*}
		
		(d) There exist $\sigma \in \mathbb{R}$ such that			\begin{align*}
			\sigma &= \sigma_{23} + i \cot \beta_1 (i\omega_{31} - i\omega_{12}) \nonumber\\
			&= \sigma_{31} + i \cot \beta_2 (i\omega_{12} - i\omega_{23}) \\
			&= \sigma_{12} + i \cot \beta_3 (i\omega_{23} - i\omega_{31}) \nonumber.
		\end{align*}
		
	\end{lemma}
	
	\begin{proof}
		The relation between (a) and (b) is obvious. We show the equivalence between (b) and (c). With $A$ denoting the signed triangle area we have the following identities:
		\begin{align*}
			0&=\langle i(z_{j}- z_i),z_{j}- z_i \rangle ,\\
			2 A &=\langle i(z_{j}- z_i),z_k- z_{j} \rangle ,\\ 
			\langle i(z_{j}- z_i),i(z_{j}- z_i) \rangle &= \langle z_{j}- z_i,z_{j}- z_i \rangle.
		\end{align*}
		Using these identities and $z_3 - z_2 \in \mbox{span}_\mathbb{R}\{i(z_1 - z_3),i(z_2 -z_1)\}$ we obtain		\begin{align}
			\label{eq:edgedecom} z_3 - z_2 &= \cot(\beta_3) i(z_2 -z_1)-\cot(\beta_2) i(z_1 - z_3) .
		\end{align}
		Cyclic permutation yields
		\begin{align*}
			z_1 - z_3 &= \cot(\beta_1) i(z_3 - z_2) - \cot(\beta_3) i(z_2 - z_1) ,  \\
			z_2 - z_1 &= \cot(\beta_2) i(z_1 - z_3) - \cot(\beta_1) i(z_3 - z_2) .
		\end{align*}
		Substituting these identities into Equation \eqref{eq:closedtri} we obtain
		\begin{align*}
			0 =& \;\phantom{+} \sigma_1 \big(\cot(\beta_3) i(z_2 - z_1)-\cot(\beta_2) i(z_1 - z_3)\big)  + \omega_{23} i(z_3 - z_2)\\
			&+ \sigma_2  \big(\cot(\beta_1) i(z_3 - z_2) - \cot(\beta_3) i(z_2 - z_1)\big) + \omega_{31} i(z_1 - z_3)\\
			&+ \sigma_3 \big(\cot(\beta_2) i(z_1 - z_3) - \cot(\beta_1) i(z_3 - z_2)\big) + \omega_{12} i(z_2 - z_1)\\
			=& \; \phantom{+} \big(\omega_1+\cot \beta_1 (\sigma_2 - \sigma_3)\big) i(z_3 - z_2)\\
			&+\big(\omega_2+\cot \beta_2(\sigma_3 - \sigma_1)\big) i(z_1 - z_3)\\
			&+\big(\omega_3+\cot \beta_3 (\sigma_1 - \sigma_2)\big) i(z_2 - z_1).
		\end{align*}
		Now we use that $\lambda_1,\lambda_2,\lambda_3 \in \mathbb{C}$ satisfy 
		\[
		\lambda_1 i(z_3 - z_2)+\lambda_2 i(z_1 - z_3)+\lambda_3 i(z_2 - z_1)=0,
		\]
		if and only if $\lambda_1=\lambda_2=\lambda_3$. This establishes the equivalence of (b) and (c). The equivalence of (b) and (d) is seen in a similar fashion by eliminating $i(z_{j}-z_i)$ in \eqref{eq:closedtri} instead of $(z_{j}-z_i)$.
		 
	\end{proof}
	
	The quantity $\omega$ above describes the average rotation speed of the triangle. Similarly, it can be verified that the above $\sigma$ satisfies
	\[
	\sigma = \frac{\dot{R}}{R}
	\]    
	where $R$ denotes the circumradius of the triangle. Thus $\sigma$ signifies an average scaling of the triangle.
	
	\subsection{Harmonic functions with respect to the cotangent Laplacian}
	\label{sec:harmonic}
	
	In smooth complex analysis conformal maps are closely related to harmonic functions. If a conformal map preserves orientation it is holomorphic and satisfies the Cauchy Riemann equations. In particular, its real part and the imaginary part are conjugate harmonic functions. Conversely, given a harmonic function on a simply connected surface then it is the real part of some conformal map.
	
	A similar relationship manifests between discrete harmonic functions (in the sense of the cotangent Laplacian) and infinitesimal deformations of triangular meshes. Discrete harmonic functions can be regarded as the real part of holomorphic functions which satisfies a discrete analogue of the Cauchy Riemann equations. In particular, a relation between discrete harmonic functions and infinitesimal pattern deformations was found by Bobenko, Mercat and Suris \cite{Bobenko2005}. Integrable systems were involved in this context. We extend their result to include the case of infinitesimal conformal deformations.
	
	\begin{theorem}
		Let $z:V \to \mathbb{C}$ be a simply connected triangular mesh realized in the complex plane and $h: V \to \mathbb{R}$ be a function. Then the following are equivalent:
		
		(a) $h$ is a harmonic function with respect to the cotangent Laplacian, i.e. using the notation of Figure \ref{fig:orientation}, for all interior vertices $i\in V_{int}$ we have
		
		\begin{equation} \label{eq:cotangent} \sum_j (\cot \beta_{ij}^k + \cot \beta_{ji}^l) (h_{j} - h_i) =0. \end{equation}
		
		(b) There exists an infinitesimal conformal deformation $\dot{z}:V \to \mathbb{C}$ with scale factors given by $h$. It is unique up to infinitesimal rotations and translations.
		
		(c) There exists an infinitesimal pattern deformation $i\dot{z}:V \to \mathbb{C}$ with $h$ as angular velocities. It is unique up to infinitesimal scalings and translations.
	\end{theorem}
	\begin{proof}
		We show the equivalence of the first two statements. The equivalence of the first and the third follows similarly.
		
		Suppose $h$ is a harmonic function. Since the triangular mesh is simply connected, equation \eqref{eq:cotangent} implies the existence of a function $\tilde{\omega}: F \to \mathbb{R}$ such that for all interior edges $\{ij\}$ we have
		\[
		i\tilde{\omega}_{ijk}- i\tilde{\omega}_{jil} = i(\cot \beta_{ij}^k + \cot \beta_{ji}^l) (h_{j} - h_i).  
		\]
		Here $\tilde{\omega}$ is unique up to an additive constant and called the \emph{conjugate harmonic function} of $h$.
		Using $\tilde{\omega}$ we define a function $\omega:E \to \mathbb{R}$ via
		\[
		i\omega_{ij} = i\tilde{\omega}_{ijk} - i\cot \beta_{ij}^k (h_{j} - h_i).
		\]
		Lemma \ref{lem:infindeform} now implies that there exists $\dot{z}: V \to \mathbb{C}$ such that
		\[
		(\dot{z}_{j} -\dot{z}_i) = \left(\frac{h_i + h_{j}}{2} + i\omega_{ij} \right) (z_{j} -z_i).
		\]
		This gives us the desired infinitesimal conformal deformation of $z$ with $h$ as scale factors.
		
		To show uniqueness, suppose $\dot{z}, \dot{z}'$ are infinitesimal conformal deformations with the same scale factors. Then $\dot{z} - \dot{z}'$ preserves all the edge lengths of the triangular mesh and hence is induced from an Euclidean transformation.

		Conversely, given an infinitesimal conformal deformation $\dot{z}$ with scale factors $h$. We write
		\[
		\dot{z}_{j} - \dot{z}_i = \left(\frac{h_i + h_{j}}{2} + i\omega_{ij} \right) (z_{j} - z_i)
		\]
		for some $\omega:E \to \mathbb{R}$. Lemma \ref{lem:infindeform} implies that there is a function $\tilde{\omega}: F \to \mathbb{R}$ such that
		\[
		i\tilde{\omega}_{ijk}=i\omega_{ij} + i\cot \beta_{ij}^k (h_{j} - h_i).
		\]
		We have
		\[
		i\tilde{\omega}_{ijk}- i\tilde{\omega}_{jil} =i (\cot \beta_{ij}^k + \cot \beta_{ji}^l) (h_{j} - h_i)
		\]
		and
		\[
		\sum_j (\cot \beta_{ij}^k + \cot \beta_{ji}^l) (h_{j} - h_i) =0. 
		\]
		Therefore $h$ is harmonic. 
	\end{proof}
	
	\section{Holomorphic quadratic differentials} \label{sec:holomorphic1forms}
	In this section, we introduce a discrete analogue of holomorphic quadratic differentials. We illustrate their correspondence to discrete harmonic functions. It reflects the property in the smooth theory that holomorphic quadratic differentials parametrize M\"{o}bius structures on Riemann surfaces (Ch. 9, \cite{Gunning1966}). 
	
	To simplify the notation, we make use of discrete differential forms . We denote $\vv{E}$ the set of oriented edges and $\vv{E}_{int}$ the set of oriented interior edges. Given an oriented triangular mesh $M$, a complex-valued function $\eta:\vv{E}\to \mathbb{C}$ is called a \emph{discrete 1-form} if
	\[
	\eta(e_{ij}) = -\eta(e_{ji}) \quad \forall e_{ij} \in \vv{E}.
	\]
	It is \emph{closed} if for every face $\{ijk\}$	\[
	\eta(e_{ij}) + \eta(e_{jk}) +\eta(e_{ki}) =0.
	\]
	It is \emph{exact} if there exists a function $f:V \to \mathbb{C}$ such that
	\[
	\eta(e_{ij}) = df(e_{ij}) := f_{j} -f_i.
	\]
	Similarly, we can consider discrete 1-forms on the dual graph $M^*$ of $M$ and these are called \emph{dual 1-forms}. Given an oriented edge $e$, we denote $e^*$ its dual edge oriented from the right face of $e$ to its left face. The set of oriented dual edges is denoted by $\vv{E}^*$.
	
	\begin{definition}
		Given a triangular mesh $z:V \to \mathbb{C}$ realized on the complex plane, a function $q:E_{int} \to i\mathbb{R}$ defined on interior edges is a \emph{discrete holomorphic quadratic differential} if it satisfies for every interior vertex $i \in V_{int}$
		\begin{gather*}
			\sum_{j} q_{ij} = 0, \\
			\sum_{j} q_{ij}/ dz(e_{ij}) = 0.
		\end{gather*}
	\end{definition}
	
	\begin{theorem}
		Let $q:E_{int} \to i\mathbb{R}$ be a holomorphic quadratic differential on a realization $z:V \to \mathbb{C}$ of a triangular mesh. Suppose $\Phi:\overline{\mathbb{C}} \to \overline{\mathbb{C}}$ is a M\"{o}bius transformation which does not map any vertex to infinity. Then $q$ is again a holomorphic quadratic differential on $w:= \Phi \circ z$.
	\end{theorem}
	\begin{proof}
		Since M\"{o}bius transformations are generated by Euclidean transformations and inversions, it suffices to consider the inversion in the unit circle at the origin
		\[
		w:=\Phi(z) = 1/z.
		\]
		We have
		\begin{align*}
			\sum_{j} q_{ij}/ dw(e_{ij}) = \sum_{j} -z_i z_{j} q_{ij}/ dz(e_{ij}) = -z_i \sum_{j} q_{ij} - z_i^2 \sum_{j} q_{ij} /dz(e_{ij}) =0.
		\end{align*}
		Hence the claims follow.  
	\end{proof}
	
	We are going to show that on a simply connected triangular mesh, there is a correspondence between discrete holomorphic quadratic differentials and discrete harmonic functions.
	
	We first show how to construct a discrete holomorphic quadratic differential from a harmonic function. Given a function $u:V \to \mathbb{R}$ on a realization of $z:V \to \mathbb{C}$ of a triangular mesh $M$. If we interpolate it piecewise-linearly over each triangular face, its gradient is constant on each face and we have $\grad_z u:F \to \mathbb{C}$ given by
	\[
	\grad_z u_{ijk} =  i \frac{u_i dz(e_{jk})+ u_{j} dz(e_{ki})+u_k dz(e_{ij})}{2 A_{ijk}}.
	\]
	Note that we ignore here the non-generic case (which leads to the vanishing of the area) where the triangle degenerates in the sense that its circumcircle passes through the point at infinity. Also note that for a non-degenerate triangle that is mapped by $z$ in $\mathbb{C}$ in an orientation reversing fashion the area $A_{ijk}$ is considered to have a negative sign. Granted this, one can verify that the gradient of $u$ satisfies
	\[
	\langle \grad_z u_{ijk}, dz(e_{ij}) \rangle = u_{j} - u_i  \quad \forall \{ij\} \subset \{ijk\} \in F.
	\]
	We define $u_z : F \to \mathbb{C}$ by
	\[
	u_z := \frac{1}{2} \overline{\grad_z u}.
	\]
	and the dual 1-form $du_z: \vv{E}^*_{int} \to \mathbb{C}$ on $M$ by
	\[
	du_z(e^*_{ij}) := (u_z)_{ijk} - (u_z)_{jil}
	\]
	where $\{ijk\}$ is the left face and $\{jil\}$ is the right face of the oriented edge $e_{ij}$.
	\begin{lemma}\label{lem:hessian}
		Given a function $u:V \to \mathbb{R}$ on a realization of a triangular mesh $z:V \to \mathbb{C}$, we have
		\begin{align*}
			& du_z(e^*_{ij}) dz(e_{ij}) \\ =& \frac{-i}{2}\big(\cot \beta^{i}_{jk} (u_k -u_{j})+\cot \beta^{j}_{ki} (u_k -u_i)+\cot \beta^{j}_{il} (u_l -u_i)+\cot \beta^{i}_{lj} (u_l -u_{j})\big).
		\end{align*}
		which is purely imaginary (Figure \ref{fig:orientation}).
	\end{lemma}
	\begin{proof}
		Since \[
		\langle \grad_z u_{ijk}, dz(e_{ij}) \rangle = u_{j} - u_i = \langle \grad_z u_{jkl}, dz(e_{ij}) \rangle,
		\]
		we have
		\[
		\Re (du_z(e^*_{ij}) dz(e_{ij})) =0.
		\]
		On the other hand, using equation \eqref{eq:edgedecom} we get
		\begin{align*}
			&\Re (du_z(e^*_{ij}) idz(e_{ij})) \\
			=& \Re (((u_z)_{ijk} - (u_z)_{jil}) idz(e_{ij})) \\
			=& (\langle \grad_z u_{ijk}, \cot \beta^i_{jk} dz(e_{jk}) - \cot \beta^{j}_{ki} dz(e_{ki}) \rangle \\&+ \langle \grad_z u_{jil}, \cot \beta^{j}_{il} dz(e_{il}) - \cot \beta^i_{lj} dz(e_{lj}) \rangle)/2 \\
			=& \frac{1}{2}\big(\cot \beta^{i}_{jk} (u_k -u_{j})+\cot \beta^{j}_{ki} (u_k -u_i)+\cot \beta^{j}_{il} (u_l -u_i)+\cot \beta^{i}_{lj} (u_l -u_{j})\big).
		\end{align*} 
		Hence the claim follows.  
	\end{proof}
	
	\begin{lemma} \label{lem:harholo}
		Given a realization $z:V \to \mathbb{C}$ of a triangular mesh. A function $u:V \to \mathbb{R}$ is harmonic if and only if the function $q:E_{int} \to i\mathbb{R}$ defined by \[
		q_{ij} := du_z(e^*_{ij}) dz(e_{ij})
		\]
		is a holomorphic quadratic differential.
	\end{lemma}
	\begin{proof}
		Note $q$ is well defined since
		\[
		q_{ij} = du_z(e^*_{ij}) dz(e_{ij}) = du_z(e^*_{ji}) dz(e_{ji}) =q_{ji}.
		\]
		It holds for general functions $u:V \to \mathbb{R}$ that
		\begin{gather*}
			\Re (q) \equiv 0 \\
			\sum_j q_{ij}/dz(e_{ij})= \sum_j du_z(e^*_{ij}) =0 \quad \forall i \in V_{int}.
		\end{gather*}
		We know from Lemma \ref{lem:hessian} that for every interior vertex $i \in V_{int}$
		\[
		\sum_j q_{ij} = \sum_j du_z(e^*_{ij}) dz(e_{ij}) = \frac{i}{2} \sum_j (\cot \beta_{ij}^k + \cot \beta_{ji}^l) (u_{j} -u_i).
		\]
		Hence, $u$ is harmonic if and only if $q$ is a holomorphic quadratic differential.   
	\end{proof}
	
	\begin{lemma} \label{lem:integrate}
		Let $z:V \to \mathbb{C}$ be a realization of a simply connected triangular mesh. Given a function $q:E_{int} \to i\mathbb{R}$ such that for every interior vertex $i \in V_{int}$ 
		\[
		\sum_j q_{ij} /dz(e_{ij}) =0,
		\]
		there exists a function $u:V \to \mathbb{R}$ such that for every interior edge $\{ij\}$
		\[
		q_{ij} = du_z(e^*_{ij}) dz(e_{ij}).
		\]
	\end{lemma}
	\begin{proof}
		We consider a dual 1-form $\tau$ on $M$ defined by
		\[
		\tau(e^*_{ij})=q_{ij} /dz(e_{ij}).
		\]
		Since $M$ is simply connected and 
		\[
		\sum_j \tau(e^*_{ij}) = \sum_j q_{ij} /dz(e_{ij}) =0,
		\]
		there exists a function $h:F \to \mathbb{C}$ such that
		\[
		dh(e^*_{ij}) := h_{ijk} - h_{jil} =\tau(e^*_{ij}).
		\]
		It implies we have $\Re(dh(e^*)dz(e)) = \Re(q) \equiv 0$ and
		\[
		\omega(e_{ij}) := \langle 2\bar{h}_{ijk} , dz(e_{ij}) \rangle = \langle 2\bar{h}_{jil} , dz(e_{ij}) \rangle .
		\]
		is a well-defined $\mathbb{R}$-valued 1-form. Since the triangular mesh is simply connected and for every face $\{ijk\}$
		\[
		\omega(e_{ij}) + \omega(e_{jk}) + \omega(e_{ki}) =0,
		\] there exists a function $u:V \to \mathbb{R}$ such that for every oriented edge $e_{ij}$
		\[
		du(e_{ij}) = u_{j} - u_i =\omega(e_{ij}).
		\]
		It can be verified that
		\[
		h = \frac{1}{2}\overline{\grad}_z u = u_z.
		\]		 
		Hence we obtain
		\[
		q_{ij} = \tau(e^*_{ij}) dz(e_{ij})= dh(e^*_{ij}) dz(e_{ij}) =du_z(e^*_{ij}) dz(e_{ij})
		\]
		for every interior edge $\{ij\}$.  
	\end{proof}
	\begin{theorem} \label{thm:holo1form}
		Suppose $z:V \to \mathbb{C}$ is a realization of a simply connected triangular mesh. Then any holomorphic quadratic differential $q: E_{int} \to i\mathbb{R}$ is of the form
		\[
		q_{ij} = du_z(e^*_{ij}) dz(e_{ij})  \quad  \forall e_{ij} \in \vv{E}_{int}
		\]
		for some harmonic function $u:V \to \mathbb{R}$.
		
		Furthermore, the space of holomorphic quadratic differentials is a vector space isomorphic to the space of discrete harmonic functions module linear functions.
	\end{theorem}
	\begin{proof}
		The first part of the statement follows from Lemma \ref{lem:harholo} and Lemma \ref{lem:integrate}. In order to show the second part, it suffices to observe that 
		\begin{align*}
			du_z \equiv 0  \iff \grad u \equiv a \iff du = \langle a, dz \rangle \iff u = \langle a,z \rangle +b 
		\end{align*}
		for some $a,b \in \mathbb{C}$.  
	\end{proof}
	
	In previous sections, we showed that every harmonic function corresponds to an infinitesimal conformal deformation. The following shows that discrete holomorphic quadratic differentials are the change in the intersection angles of circumscribed circles.
	\begin{theorem}\label{thm:hesscr}
		Let $z:V \to \mathbb{C}$ be a realization of a simply connected triangular mesh. Suppose $u:V \to \mathbb{R}$ is a discrete harmonic function and $\dot{z}$ is an infinitesimal conformal deformation with $u$ as scale factors. Then we have
		\[
		du_z dz = -\frac{1}{2} \frac{\dot{\cratio}_{z}}{\cratio_z} = -\frac{i}{2} \dot{\phi}
		\]
		where $\dot{\phi}:E_{int} \to \mathbb{R}$ denotes the change in the intersection angles of neighboring circumscribed circles.  
	\end{theorem}
	\begin{proof}
		We write $(\dot{z}_{j} -\dot{z}_i) = (\frac{h_i + h_{j}}{2} + i\omega_{ij}) (z_{j} -z_i)$. Applying Lemma \ref{lem:hessian} we have
		\begin{align*}
			\dot{\cratio}_{z,ij}/\cratio_{z,ij} =& i\omega_{jk} - i\omega_{ki} + i\omega_{il} - i\omega_{lj} \\
			=& 	i\big(\cot \beta^{i}_{jk} (u_k -u_{j})+\cot \beta^{j}_{ki} (u_k -u_i)+\cot \beta^{j}_{il} (u_l -u_i)+\cot \beta^{i}_{lj} (u_l -u_{j})\big) \\
			=& 	-2 du_z(e^*_{ij}) dz(e_{ij}).
		\end{align*}
		The equality
		\[
		\frac{\dot{\cratio}_{z}}{\cratio_z} = i\dot{\phi}
		\]
		follows from Equation \eqref{eq:intersectionangles}.  
	\end{proof}
	
	\section{Conformal deformations in terms of End($\mathbb{C}^2$)}
	In this section we show how an infinitesimal conformal deformation gives rise to a discrete analogue of a holomorphic null curve in $\mathbb{C}^3$. Later we will see that the real parts of such a "holomorphic null curve" can be regarded as the Weierstrass representation of a discrete minimal surface.
	
	Up to now we have mostly treated the Riemann sphere $\mathbb{C}\textrm{P}^1$ as the extended complex plane $\overline{\mathbb{C}}=\mathbb{C}\cup \{ \infty \}$. In this section we will take a more explicitly M\"{o}bius geometric approach: We will represent fractional linear transformations of $\overline{\mathbb{C}}$ by linear transformations of $\mathbb{C}^2$ with determinant one. Actually, the group of M\"{o}bius transformations is
	\[
	\text{M\"{o}b}(\overline{\mathbb{C}}) \cong \mbox{PSL}(2,\mathbb{C}) \cong \mbox{SL}(2,\mathbb{C})/(\pm I).
	\]
	However, since we are mainly interested in infinitesimal deformations and any map into $\mbox{PSL}(2,\mathbb{C})$ whose values stays close to the identity admits a canonical lift to $\mbox{SL}(2,\mathbb{C})$, we can safely ignore the difference between $\mbox{PSL}(2,\mathbb{C})$ and $\mbox{SL}(2,\mathbb{C})$.
	
	Given a realization $z:V \to \mathbb{C}$ of a triangular mesh we consider its lift $\psi:V \to \mathbb{C}^2$
	\[
	\psi := \left(\begin{array}{c} z \\ 1
	\end{array} \right)
	\]
	and regard the realization as a map $\Psi:V\to\mathbb{C}\textrm{P}^1$ where
	\begin{equation*}
		\Psi :=\mathbb{C}\left(\begin{array}{c} z \\ 1
		\end{array} \right)=[\psi].
	\end{equation*}
	
	The action of a M\"{o}bius transformation on the Riemann sphere is given by a matrix $A \in \mbox{SL}(2,\mathbb{C})$, which is unique up to sign:
	\[
	[\varphi] \mapsto [A\varphi].
	\]
	
	Before we investigate infinitesimal deformations we first consider finite deformations of a triangular mesh $\Psi:V \to \mathbb{C}\textnormal{P}^1$. Given such a finite deformation, the change in the positions of the three vertices of a triangle $\{ijk\}$ can be described by a M\"{o}bius transformation, which is represented by $G_{ijk} \in \mbox{SL}(2,\mathbb{C})$. They satisfy a compatibility condition on each interior edge $\{ij\}$ (see Figure \ref{fig:orientation}):
	\begin{gather*}
		[G_{ijk}\psi_i] = [G_{jil}\psi_i], \\ [G_{ijk}\psi_{j}] = [G_{jil}\psi_{j}].
	\end{gather*}
	Suppose now that the mesh is simply connected. Then up to a global M\"{o}bius transformation the map $G:F \to \mbox{SL}(2,\mathbb{C})$ can be uniquely reconstructed from the \emph{multiplicative dual 1-form} defined as
	\begin{equation*}
		G(e^*_{ij}):=G_{jil}^{-1}G_{ijk}.
	\end{equation*}	
	$G(e^*_{ij})$ is defined whenever $\{ij\}$ is an interior edge and we have
	\begin{equation*}
		G(e^*_{ij}) = G(e^*_{ji})^{-1}.
	\end{equation*}
	Moreover, for every interior vertex $i$ we have
	\begin{equation*}
		\prod_j G(e^*_{ij}) = I.
	\end{equation*}
	The compatibility conditions imply that for interior each edge $\{ij\}$ there exist $\lambda_{ij,i},\lambda_{ij,j} \in \mathbb{C} \backslash \{0\}$ such that
	\begin{align*}
		G(e^*_{ij}) \psi_i &= \lambda_{ij,i} \psi_i \\ G(e^*_{ij}) \psi_{j} &= \lambda_{ij,j} \psi_{j}.
	\end{align*}
	Since $\lambda_{ij,i}\,\lambda_{ij,j} = \det(G(e^*_{ij}))=1$, we have 
	\[
	\lambda_{ij} := \lambda_{ij,i} = 1/ \lambda_{ij,j}.
	\]
	Because of $G(e^*_{ij}) = G(e^*_{ij})^{-1}$ we know
	\[
	\lambda_{ij} = \lambda_{ij,i} = 1/ \lambda_{ji,i} = \lambda_{ji}.
	\]
	Hence $\lambda$ defines a complex-valued function on the set $E_{int}$ of interior edges.
	
	We now show that for each interior edge $\lambda_{ij}$ determines the change in the cross ratio of the four points of the two adjacent triangles. Note that the cross ratio of four points in $\mathbb{C}\textrm{P}^1$ can expressed as
	\[
	\cratio([\psi_{j}],[\psi_k],[\psi_i],[\psi_l]) = \frac{\det(\psi_k,\psi_{j})\det(\psi_l,\psi_i)}{\det(\psi_i,\psi_k)\det(\psi_{j},\psi_l)}.
	\]
	\begin{lemma}
		Suppose we are given four points $[\psi_i],[\psi_{j}],[\psi_k],[\psi_l] \in \mathbb{C}\textnormal{P}^1$ and $G \in \textnormal{SL}(2,\mathbb{C})$ with
		\begin{align*}
			G \psi_i &= \lambda^{-1}\psi_i \\
			G \psi_{j} &= \lambda \psi_{j}
		\end{align*}
		for some $\lambda \in \mathbb{C} \backslash \{0\}$. Then the cross ratio of the four transformed points
		\begin{equation*}
			[\tilde{\psi}_i]=[G \psi_i] \quad,\quad
			[\tilde{\psi}_{j}]=[G \psi_{j}] \quad,\quad
			[\tilde{\psi}_k]=[G \psi_k] \quad,\quad
			[\tilde{\psi}_l]=[\psi_l]
		\end{equation*}
		is given by
	\end{lemma}
	\begin{equation*}
		\cratio([\tilde{\psi}_{j}],[\tilde{\psi}_k],[\tilde{\psi}_i],[\tilde{\psi}_l]) =  \cratio([\psi_{j}],[\psi_k],[\psi_i],[\psi_l])/\lambda^2.
	\end{equation*}
	
	\begin{proof}
		\begin{align*}
			\cratio([\tilde{\psi}_{j}],[\tilde{\psi}_k],[\tilde{\psi}_i],[\tilde{\psi}_l])&=\frac{\det(G\psi_k,G\psi_{j})\det(\psi_l,G \psi_i)}{\det(G\psi_i,G\psi_k)\det(G\psi_{j},\psi_l)} \\
			&=	\cratio([\psi_{j}],[\psi_k],[\psi_i],[\psi_l])/\lambda^2.
		\end{align*}
		 
	\end{proof}
	
	We now can summarize the information about finite deformations of a realization as follows:
	
	\begin{theorem}
		Let $\Psi:V \to \mathbb{C}\textnormal{P}^1$ be a realization of a simply connected triangular mesh. Then there is a bijection between finite deformations of $\Psi$ in $\mathbb{C}\textnormal{P}^1$ modulo global M\"{o}bius transformations and multiplicative dual 1 forms $G: \vv{E}^*_{int} \to \textnormal{SL}(2,\mathbb{C})$ satisfying for every interior vertex $i$
		\begin{equation*}
			\prod_j G(e^*_{ij}) = I
		\end{equation*}
		and for every interior edge
		\begin{align*}
			G(e^*_{ij}) &= G(e^*_{ji})^{-1} \\
			G(e^*_{ij}) \psi_i &= \lambda_{ij}^{-1}\psi_i \\
			G(e^*_{ij}) \psi_{j} &= \lambda_{ij} \psi_{j} .
		\end{align*}
		Here $\lambda:E_{int} \to \mathbb{C} \backslash \{0\}$. We denote by $\cratio:E_{int} \to \mathbb{C}$ the cross ratios of $\Psi$ and $\widetilde{\cratio}:E_{int} \to \mathbb{C}$ the cross ratios of a new realization described by $G$. Then
		\[
		\widetilde{\cratio}= \cratio/\lambda^2.
		\]
		In particular, 
		\begin{align*}
			|\lambda| \equiv 1 & \implies \text{ the deformation is conformal}. \\
			\Arg(\lambda) \equiv 0 & \implies \text{  the deformation is a pattern deformation.} 
		\end{align*}
	\end{theorem}
	
	Suppose we have a family of deformations described by dual 1-forms $G_t: \vv{E}_{int} \to \mbox{SL}(2,\mathbb{C})$ with $G_0 \equiv I$. By considering $\eta := \frac{d}{dt}|_{t=0}\, G_t$ we obtain the following description of infinitesimal deformations:
	\begin{corollary}
		
		Let $\Psi:V \to \mathbb{C}\textnormal{P}^1$ be a realization of a simply connected triangular mesh. Then there is a bijection between infinitesimal deformations of $\Psi$ in $\mathbb{C}\textnormal{P}^1$ modulo infinitesimal M\"{o}bius transformations and dual 1 forms $\eta: \vv{E}_{int} \to \textnormal{sl}(2,\mathbb{C})$ satisfying for every interior vertex $i$
		\begin{equation}
			\sum_j \eta(e^*_{ij}) = 0 \label{eq:lieclosed}
		\end{equation}
		and for every interior edge
		\begin{align*}
			\eta(e^*_{ij}) &= -\eta(e^*_{ji}) \\
			\eta(e^*_{ij}) \psi_i &= -\mu_{ij}\,\psi_i \\
			\eta(e^*_{ij}) \psi_{j} &= \mu_{ij} \,\psi_{j}.
		\end{align*}
		Here $\mu: E_{int} \to \mathbb{C}$. We denote by $\cratio:E_{int} \to \mathbb{C}$ the cross ratios of $\Psi$ and $\dot{\cratio}:E_{int} \to \mathbb{C}$ the rate of change in cross ratios induced by the infinitesimal deformation described by $\eta$. Then
		\[
		\mu = -\frac{1}{2} \frac{\dot{\cratio}}{\cratio}.
		\]
		In particular,
		\begin{align*}
			\Re(\mu) \equiv 0 & \implies \text{ the infinitesimal deformation is conformal}, \\
			\Im(\mu) \equiv 0 & \implies \text{ the infinitesimal deformation is a pattern deformation.}
		\end{align*}
	\end{corollary}
	
	Note that given a mesh, the 1-form $\eta$ is uniquely determined by the eigenfunction $\mu$. We now investigate the constraints on $\mu$ implied by the closedness condition \eqref{eq:lieclosed} of $\eta$. 
	
	Consider the symmetric bilinear form $(\,,): \mathbb{C}^2 \times \mathbb{C}^2 \to \mbox{sl}(2,\mathbb{C})$
	\begin{align*}
		(\phi, \varphi)v := \det(\phi,v) \varphi + \det(\varphi,v) \phi.
	\end{align*}
	For $\psi_i \neq \psi_{j} \in \mathbb{C}^2$ we define
	\[
	m_{ij}:= \frac{1}{\det(\psi_i,\psi_{j})} (\psi_{j}, \psi_i) \in \textnormal{sl}(2,\mathbb{C}).
	\]
	The matrix $m_{ij}$ is independent of the representatives of $[\psi_i],[\psi_{j}] \in \mathbb{C}\textnormal{P}^1$ and we have
	\begin{align*}
		m_{ij} &= -m_{ji} \\
		m_{ij} \psi_i &=-\psi_i \\
		m_{ij} \psi_{j} &=\psi_{j}.
	\end{align*}
	Using the representatives $\psi_i = \left( \begin{array}{c}
	z_i \\ 1
	\end{array}\right)$ we obtain
	\begin{align*}
		\eta(e^*_{ij}) =& \frac{\mu_{ij}}{\det(\psi_{j},\psi_i)} (\psi_i, \psi_{j}) \\
		=& \frac{\mu_{ij}}{z_{j} -z_i} \left( \begin{array}{cc} z_i +z_{j} & -2z_i z_{j} \\ 2 & -z_i - z_{j} \end{array} \right).
	\end{align*}
	Hence
	\begin{equation}\label{eq:null-curve}
		\sum_j \eta(e^*_{ij}) =0  \quad \iff \quad \sum_j \mu_{ij} =0 \quad \text{and} \quad \sum_j \mu_{ij}/ (z_{j}-z_i) =0. 
	\end{equation}
	We consider the Pauli matrices
	\begin{align*}
		\sigma_1 = \left( \begin{array}{cc} 0 & 1 \\ 1 & 0 \end{array} \right)\quad ,\quad \sigma_2 = \left( \begin{array}{cc} 0 & i \\ -i & 0 \end{array} \right)\quad ,\quad \sigma_3 = \left( \begin{array}{cc} 1 & 0 \\ 0 & -1 \end{array} \right)
	\end{align*}
	which form a basis of $\textnormal{sl}(2,\mathbb{C})$. Then
	\[
	\eta(e^*_{ij}) = \frac{\mu_{ij}}{z_{j} -z_i}( (1-z_iz_{j})\sigma_1 + i (1+z_iz_{j}) \sigma_2 + (z_i+z_{j}) \sigma_3).
	\]
	If we now identify $\textnormal{sl}(2,\mathbb{C})$ with $\mathbb{C}^3$ via
	\[
	\sigma_i \mapsto  \left( \begin{array}{c} 1 \\ 0 \\ 0 \end{array} \right),\sigma_2 \mapsto  \left( \begin{array}{c} 0 \\ 1 \\ 0 \end{array} \right),\sigma_3 \mapsto  \left( \begin{array}{c} 0 \\ 0 \\ 1 \end{array} \right),
	\] 
	we obtain
	\begin{equation}
		\label{eq:weierstrass-one-form}
		\eta(e^*_{ij}) = \frac{\mu_{ij}}{z_{j} -z_i} \left( \begin{array}{c} 1-z_iz_{j} \\ i (1+z_iz_{j}) \\ z_i+z_{j} \end{array} \right).
	\end{equation}
	Thus to every infinitesimal deformation of a realized triangular mesh we can associate a closed $\textnormal{sl}(2,\mathbb{C})$-valued dual 1-form. In the special case of an infinitesimal conformal deformation (i.e. $\mu$ is real-valued) we will see that this yields a discrete analogue of the Weierstrass representation for minimal surfaces.
	
	\section{Weierstrass representation of discrete minimal surfaces}
	
	The Weierstrass representation for minimal surfaces in $\mathbb{R}^3$ is the most classical example for applications of complex analysis:
	
	\begin{theorem}
		Given two meromorphic functions $g,h: U \subset \mathbb{C} \to \mathbb{C}$ such that $g^2 h$ is holomorphic. Then $f:U \to \mathbb{R}^3$ defined by
		\begin{equation*}\label{eq:weierstrass}
			df = \Re \left( \left( \begin{array}{c}
				1-g^2 \\ i(1+g^2) \\ 2g 
			\end{array}\right) h(z) dz \right) = \Re \left( \left( \begin{array}{c}
			1-g^2 \\ i(1+g^2) \\ 2g 
		\end{array}\right) \frac{q}{dg} \right)
	\end{equation*}
	is a minimal surface. Its Gau{\ss} map $n$ is the stereographic projection of $g$
	\begin{equation*}
		n=\frac{1}{|g|^2+1}\left(\begin{array}{c}2\Re g \\ 2 \Im g\\ |g|^2-1 \end{array}\right).
	\end{equation*}
	The holomorphic quadratic differential $q:= hg_z dz^2$ is called the Hopf differential of $f$ and encodes its second fundamental form: The direction defined by a nonzero tangent vector $W$ is
	\begin{equation*}
		\begin{array}{rcl}
			\text{an asymptotic direction} & \iff & q(W) \in i \mathbb{R}. \\
			\text{a principal curvature direction} & \iff & q(W) \in  \mathbb{R}.
		\end{array}	
	\end{equation*}		
	Locally, every minimal surface can be written in this form.
\end{theorem}

\begin{figure}[h]
	\centering
	\includegraphics[width=0.9\textwidth]{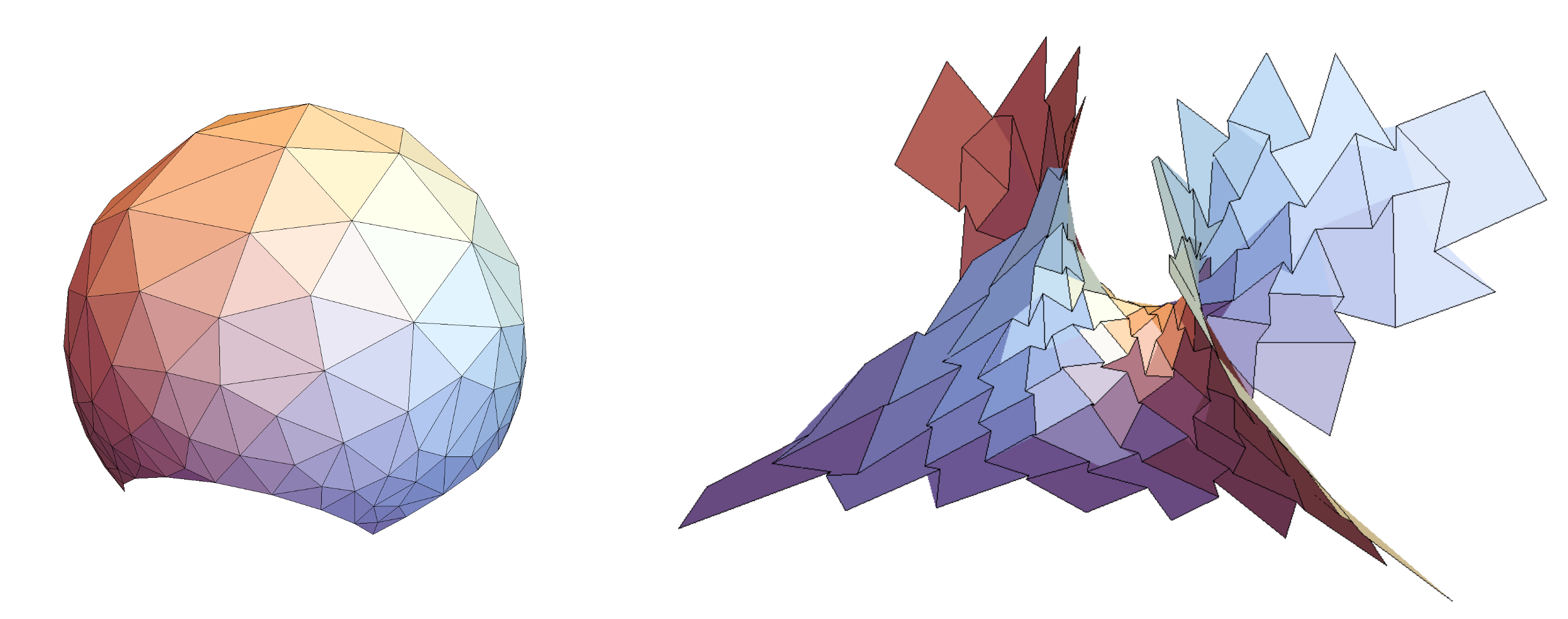}
	\caption{Left: a triangulated surface $n:V \to \mathbb{S}^2$ with vertices on the unit sphere. Right: a discrete minimal surface $f:F \to \mathbb{R}^3$ satisfying Definition \ref{def:discretemin}.}
	\label{fig:minimal}
\end{figure}

We now develop a discrete version of this theorem for arbitrary triangular meshes realized in the complex plane. A similar formula for quadrilateral meshes with factorized real cross ratios was established by Bobenko and Pinkall \cite{Bobenko1996}. Here we will use the definition of a discrete minimal surface $f$ with Gau{\ss} map $n$ given in \cite{Lam2015}:

\begin{definition} \label{def:discretemin}
	Let $n: V\to \mathbb{S}^2$ be a realization of a triangular mesh on the unit sphere in $\mathbb{R}^3$. Then a map $f:F \to \mathbb{R}^3$ defined on the set $F$ of faces is called a discrete minimal surface with Gau{\ss} map $n$ if for all oriented interior edges $e_{ij}$ we have
	\[
	(n_{j} - n_i) \times (f_{ijk} - f_{jil}) =0.
	\]
	Here $\{ijk\}$ and $\{jil\}$ denote the left and the right faces of $e_{ij}$.
\end{definition}

This definition mirrors the fact from the smooth theory that minimal surfaces are Christoffel duals of their Gau{\ss} maps (Figure \ref{fig:minimal}). The correspondence between discrete harmonic functions and discrete minimal surfaces was observed in \cite{Lam2015}. Here is a Weierstrass representation for discrete minimal surfaces in terms of their Gau{\ss} map and their Hopf differential:

\begin{theorem}
	Let $z:V \to \mathbb{C}$ be a realization of a simply connected triangular mesh and $q:E_{int} \to i\mathbb{R}$ a holomorphic quadratic differential. Then there exists $f:F \to \mathbb{R}^3$ such that for every interior edge $\{ij\}$
	\begin{equation} \label{eq:eta}
		df(e^*_{ij}) = \Re\left( \frac{q_{ij}}{i(z_{j} - z_i)}  \left( \begin{array}{c}
			1-z_i z_{j} \\ i(1+z_i z_{j}) \\ z_i + z_{j}
		\end{array}\right) \right).
	\end{equation}
	Moreover $f$ is a discrete minimal surface with Gau{\ss} map
	\begin{equation*}
		n = \frac{1}{|z|^2+1}\left(\begin{array}{c}2\Re z \\ 2\Im z \\ |z|^2-1\end{array}\right).
	\end{equation*}
	Locally, every discrete minimal surface can be written in this form.
\end{theorem}
\begin{proof}
	Suppose $q:E_{int} \to i\mathbb{R}$ is a holomorphic quadratic differential. Then by \eqref{eq:null-curve} and \eqref{eq:weierstrass-one-form} the dual 1-form $\eta$ defined as
	\[
	\eta(e^*_{ij}) :=  \frac{q_{ij}}{i(z_{j} - z_i)}  \left( \begin{array}{c}
	1-z_i z_{j} \\ i(1+z_i z_{j}) \\ z_i + z_{j}
	\end{array}\right) 
	\]
	satisfies
	\[
	\sum_j \eta(e^*_{ij}) =0.
	\]
	for all interior vertices $i$. Therefore, since the triangular mesh is simply connected, there exists $\mathfrak{F}:F \to \mathbb{C}^3$ such that for any interior edge $e$ we have
	\[
	d\mathfrak{F}(e^*) = \eta(e^*).
	\]
	Thus the map $f:F \to \mathbb{R}^3$ defined by $f:= \Re \mathfrak{F}$ satisfies Equation \eqref{eq:eta}. To show that $f$ is a discrete minimal surface we define a function $k:E_{int} \to \mathbb{R}$ by
	\[
	k_{ij} := -i \,q_{ij}/|z_{j} - z_i|^2.
	\]
	Then by direct computation we obtain
	\begin{equation} \label{eq:kminimal}
		df(e^*_{ij}) = \frac{k_{ij} (1+|z_i|^2) (1+|z_{j}|^2)}{2}  (n_{j} - n_i).
	\end{equation}
	This shows that $f$ is a discrete minimal surface with Gau{\ss} map $n$. The converse is straightforward: Given a discrete minimal surface $f$ with Gau{\ss} map $n$ we define $k:E_{int} \to \mathbb{R}$ via \eqref{eq:kminimal}. Then it can be shown that the function
	\[
	q_{ij}:=  i\,k_{ij} |z_{j}-z_i|^2
	\]
	is a holomorphic quadratic differential.  
\end{proof}

\begin{remark}
	The discrete minimal surfaces given by \eqref{eq:eta} are trivalent meshes with planar vertex stars for purely imaginary $q$. It is closely related to discrete asymptotic nets. The factor $i$ in front of $z_{j}-z_i$ appears since the integration is taken over a dual mesh while in the smooth theory $*dz=idz$.
\end{remark}

Note that we could also consider the periodic one-parameter family of maps $f^\alpha:F \to \mathbb{R}^3$ defined for $\alpha \in \mathbb{R}$ by
\[
f^{\alpha}:= \Re (e^{i\alpha} \mathfrak{F}).
\]
This family of discrete surfaces can be regarded as an associate family of minimal surfaces and is investigated in \cite{Lam2015b}.
	
	\appendix
	
	\bibliographystyle{abbrv}
	\bibliography{holomorphic_quadratic}

\end{document}

%% file: triangles.pdf_tex
%% Creator: Inkscape 0.48.5, www.inkscape.org
%% Accompanies image file 'triangles.pdf' (pdf, eps, ps)
\begingroup%
  \makeatletter%
  \providecommand\color[2][]{%
    \errmessage{(Inkscape) Color is used for the text in Inkscape, but the package 'color.sty' is not loaded}%
    \renewcommand\color[2][]{}%
  }%
  \providecommand\transparent[1]{%
    \errmessage{(Inkscape) Transparency is used (non-zero) for the text in Inkscape, but the package 'transparent.sty' is not loaded}%
    \renewcommand\transparent[1]{}%
  }%
  \providecommand\rotatebox[2]{#2}%
  \ifx\svgwidth\undefined%
    \setlength{\unitlength}{788.065625bp}%
    \ifx\svgscale\undefined%
      \relax%
    \else%
      \setlength{\unitlength}{\unitlength * \real{\svgscale}}%
    \fi%
  \else%
    \setlength{\unitlength}{\svgwidth}%
  \fi%
  \global\let\svgwidth\undefined%
  \global\let\svgscale\undefined%
  \makeatother%
  \begin{picture}(1,0.58920403)%
    \put(0,0){\includegraphics[width=\unitlength]{triangles.pdf}}%
    \put(0.49202382,0.00292845){\color[rgb]{0,0,0}\makebox(0,0)[lb]{\smash{$i$}}}%
    \put(0.49350694,0.56698388){\color[rgb]{0,0,0}\makebox(0,0)[lb]{\smash{$j$}}}%
    \put(-0.00101316,0.24221436){\color[rgb]{0,0,0}\makebox(0,0)[lb]{\smash{$k$}}}%
    \put(0.96337352,0.22046126){\color[rgb]{0,0,0}\makebox(0,0)[lb]{\smash{$l$}}}%
    \put(0.40620454,0.42058959){\color[rgb]{0,0,0}\makebox(0,0)[lb]{\smash{$\beta_{ki}^{j}$}}}%
    \put(0.4023989,0.14505058){\color[rgb]{0,0,0}\makebox(0,0)[lb]{\smash{$\beta_{jk}^{i}$}}}%
    \put(0.12801999,0.25091558){\color[rgb]{0,0,0}\makebox(0,0)[lb]{\smash{$\beta_{i\!j}^{k}$}}}%
    \put(0.79163408,0.23641353){\color[rgb]{0,0,0}\makebox(0,0)[lb]{\smash{$\beta_{ji}^{l}$}}}%
    \put(0.53320747,0.14215019){\color[rgb]{0,0,0}\makebox(0,0)[lb]{\smash{$\beta_{l\!j}^{i}$}}}%
    \put(0.52595642,0.4176892){\color[rgb]{0,0,0}\makebox(0,0)[lb]{\smash{$\beta_{il}^{j}$}}}%
  \end{picture}%
\endgroup%

%% file: angle.pdf_tex
%% Creator: Inkscape 0.48.5, www.inkscape.org
%% Accompanies image file 'angle.pdf' (pdf, eps, ps)
\begingroup%
  \makeatletter%
  \providecommand\color[2][]{%
    \errmessage{(Inkscape) Color is used for the text in Inkscape, but the package 'color.sty' is not loaded}%
    \renewcommand\color[2][]{}%
  }%
  \providecommand\transparent[1]{%
    \errmessage{(Inkscape) Transparency is used (non-zero) for the text in Inkscape, but the package 'transparent.sty' is not loaded}%
    \renewcommand\transparent[1]{}%
  }%
  \providecommand\rotatebox[2]{#2}%
  \ifx\svgwidth\undefined%
    \setlength{\unitlength}{1000bp}%
    \ifx\svgscale\undefined%
      \relax%
    \else%
      \setlength{\unitlength}{\unitlength * \real{\svgscale}}%
    \fi%
  \else%
    \setlength{\unitlength}{\svgwidth}%
  \fi%
  \global\let\svgwidth\undefined%
  \global\let\svgscale\undefined%
  \makeatother%
  \begin{picture}(1,0.6)%
    \put(0,0){\includegraphics[width=\unitlength]{angle.pdf}}%
    \put(0.44677162,0.05914347){\color[rgb]{0,0,0}\makebox(0,0)[lb]{\smash{$i$}}}%
    \put(0.44586439,0.35753357){\color[rgb]{0,0,0}\makebox(0,0)[lb]{\smash{$j$}}}%
    \put(0.02483113,0.20524474){\color[rgb]{0,0,0}\makebox(0,0)[lb]{\smash{$k$}}}%
    \put(0.98246014,0.29635553){\color[rgb]{0,0,0}\makebox(0,0)[lb]{\smash{$l$}}}%
    \put(0.4046392,0.50143905){\color[rgb]{0,0,0}\makebox(0,0)[lb]{\smash{$\phi_{i\!j}$}}}%
  \end{picture}%
\endgroup%